\documentclass[a4paper,11pt]{amsart}
\usepackage{amsfonts,amsthm,amsmath,amssymb,graphicx}
\theoremstyle{plain}
\newtheorem{lem}{Lemma}[section]
\newtheorem{prop}[lem]{Proposition}
\newtheorem{thm}[lem]{Theorem}

\theoremstyle{definition}

\theoremstyle{remark}

\DeclareMathOperator{\cls}{cls}

\DeclareMathOperator{\disc}{disc}

\DeclareMathOperator{\sym}{sym}

\DeclareMathOperator{\e}{e}

\DeclareMathOperator{\diag}{diag}

\newcommand{\calS}{{\mathcal S}}

\newcommand{\Z}{\mathbb Z}
\newcommand{\Q}{\mathbb Q}

\newcommand{\R}{\mathbb R}
\newcommand{\C}{\mathbb C}
\newcommand{\stufe}{\mathcal N }

\newcommand{\h}{\mathfrak h}

\begin{document}

% Enter full title and short title for running headers
\title[Some relations on Fourier coefficients]
{Some relations on Fourier coefficients of degree 2 Siegel forms of arbitrary level}

% Author name(s)
\author{Lynne H. Walling}
\address{School of Mathematics, University of Bristol, University Walk, Clifton, Bristol BS8 1TW, United Kingdom;
phone +44 (0)117 331-5245, fax +44 (0)117 928-7978}
\email{l.walling@bristol.ac.uk}
% Abbreviated author name for running headers
%\abbrevauthor{L.H. Walling}
% Abbreviated author name for first page header
%\headabbrevauthor{Walling, L.H.}

%\address{School of Mathematics, University of Bristol, University Walk, Clifton, Bristol BS8 1TW, United Kingdom}

% Address / e-mail address of corresponding author
%\correspdetails{l.walling@bristol.ac.uk}


\begin{abstract} We extend some recent work of D. McCarthy, proving relations among some Fourier coefficients of a degree 2
Siegel modular form $F$ with  arbitrary level and character, provided there are some primes $p$ so that $F$ is an eigenform for the Hecke operators $T(p)$ and $T_1(p^2)$.
\end{abstract}

\maketitle
\def\thefootnote{}
\def\thefootnote{\arabic{footnote}}

\section{Introduction} 
In a recent paper \cite{M}, McCarthy derives some nice results for Fourier coefficients and Hecke eigenvalues of degree 2 Siegel modular forms of level 1, extending some classical results regarding elliptic modular forms.  In particular,  with $F$ a degree 2, level 1 Siegel modular form that is an eigenform for all the Hecke operators $T(p)$, $T(p^2)$ ($p$ prime), and $a(T)$ denoting the $T$th Fourier coefficient of $F$, 
McCarthy shows that:
\begin{enumerate}
\item[(a)] 
provided that $a(I)=1$ and $p$ is prime, the $T(p)$-eigenvalue $\lambda(p)$ and the $T(p^2)$-eigenvalue $\lambda(p^2)$ are described explicitly in terms of $a(pI)$ and $a(p^2I)$;
\item[(b)] for $r\ge1$, $a(I)a(p^{r+1}I)$ is described explicitly in terms of $a(I)$, $a(pI)$, $a(p^{r-1}I)$, 
$a\begin{pmatrix}p^{r-1}\\&p^{r+1}\end{pmatrix},$
and 
$a\left(p^r\begin{pmatrix}(1+u^2)/p&u\\u&p\end{pmatrix}\right)$
where $1\le u<p/2$ with $u^2\not\equiv1\ (p)$;
\item[(c)]  if $a(I)=0$ then $a(mI)=0$ for all $m\in\Z_+$;
further, if $m,n\in\Z_+$ with $(m,n)=1$, then $a(I)a(mnI)=a(mI)a(nI).$
\end{enumerate}
(As defined in Sec. 2, 
$T_2(p^2)$ is the Hecke operator associated with the matrix $\diag(p,p,1/p,1/p)$,
 $T_1(p^2)$ is the Hecke operator associated
 with the matrix $\diag(p,1,1/p,1)$, and $T(p^2)=T_2(p^2)+p^{k-3}T_1(p^2)+p^{2k-6}$.  In \cite{HW}, for $\chi=1$,$T(p^2)$ is denoted by $\widetilde T_2(p^2)$.)
McCarthy's approach begins with some formulas from \cite{And}, which are somewhat cumbersome.

In this note we use the formulas from \cite{HW} that give the action of Hecke operators on Fourier coefficients of a Siegel modular form $F$, allowing for arbitrary level and character, and giving a simpler proof of McCarthy's above results (with no restriction on the level or character). 
Here when we say that a modular form has weight $k$, level $\stufe$ and character $\chi$, we mean that it transforms with weight $k$ and character $\chi$ under the congruence subgroup 
$$\Gamma_0(\stufe)=\left\{\begin{pmatrix}A&B\\C&D\end{pmatrix}\in Sp_2(\Z):\ \stufe|C\ \right\},$$
where $Sp_2(\Z)$ is the symplectic group of $4\times 4$ integral matrices.
We work with ``Fourier coefficients" attached to lattices (as explained below), making it simpler to work with the image of $F$ under a Hecke operator.
% We also relax our conditions on $F$, supposing it is an eigenform for at least some of % the local Hecke algebras.
For $p$ prime and degree 2, the local Hecke algebra is generated by $T(p)$, $T_1(p^2)$ and $T_2(p^2)$.
When $\stufe=1$, Proposition 5.1 of \cite{HW} gives a relation between these generators, from which we deduce that with  $p\nmid\stufe$, $T(p)$ and $T_1(p^2)$ generate the local Hecke algebra, as do $T(p)$ and $\widetilde T_2(p^2)$. 
However, when $p|\stufe$, we have $T_2(p^2)=(T(p))^2$.
Hence in this note we use the local generators $T(p)$ and $T_1(p^2)$;
to more easily apply the results of \cite{HW}, 
we use the operator 
$$\widetilde T_1(p^2)=T_1(p^2)+\chi(p)p^{k-3}(p+1)$$
 in place of $T_1(p^2)$.
% (See Sec. 2 for definitions of these operators.)

Using some rather special aspects of working with degree 2 Siegel modular forms, we 
prove the following extensions of \cite{M}.

\begin{thm}  Suppose that $F$ is a degree 2 Siegel modular form of weight $k\in\Z_+$,  level $\stufe$ and character $\chi$ with Fourier expansion
$$F(\tau)=\sum_T a(T) \exp(2\pi iTr(T\tau)).$$
Also suppose that $p$ is prime with $F|T(p)=\lambda(p)F$ and
$F|\widetilde T_1(p^2)=\widetilde\lambda_1(p^2)F$.

\begin{enumerate}

\item[(a)]  
We have
$$
\lambda(p)a(mI)
=\chi(p)p^{k-2}\eta(p)a(mI)+a(mpI),
$$
where 
$$\eta(p)=
\begin{cases}
1+\chi(-1)(-1)^k&\text{if $p\equiv 1\ (4)$,}\\
0&\text{if $p\equiv 3\ (4)$,}\\
1&\text{if $p=2.$}
\end{cases}$$
(Thus when $a(mI)\not=0$, $\lambda(p)$ is given explicitly in terms of $p$, $a(mI)$ and $a(pmI)$.)
As well, we have
\begin{align*}
\chi(p)p^{k-2}\widetilde\lambda_1(p^2)a(mI)
&= \chi(p^2)p^{2k-4}(\alpha(I;p)-p)a(mI)\\
&\quad + \lambda(p)a(pmI)-a(p^2mI)
\end{align*}
where 
$$\alpha(I;p)=\begin{cases}
2&\text{if $p\equiv 1\ (4)$,}\\
0&\text{if  $p\equiv 3\ (4)$,}\\
1&\text{if  $p=2$.}
\end{cases}$$ 
(Thus when $\chi(p)a(mI)\not=0$,  $\widetilde\lambda_1(p^2)$ is given explicitly in terms of $p$, $a(mI)$, $a(pmI)$ and $a(p^2mI)$.)

\item[(b)]  Set $\epsilon=1+\chi(-1)(-1)^k$.
For $r\ge1$, $a(mI)a(p^{r+1}I)$ is given by
\begin{align*}
&a(pmI)a(p^rI)-\chi(p^2)p^{2k-3}a(mI)a(p^{r-1}I)\\
&\quad+
\epsilon\chi(p)p^{k-2}a(mI) a\begin{pmatrix}p^{r-1}m\\&p^{r+1}m\end{pmatrix}\\
&\quad +
\epsilon\chi(p)p^{k-2}a(mI)
\sum_{\substack{1\le u<p/2\\ u^2\not\equiv-1\,(p)}}
a\left(p^rm\begin{pmatrix}(1+u^2)/p&u\\u&p\end{pmatrix}\right).
\end{align*}

\item[(c)]  Suppose that 
$n$ a product of powers of primes 
$p$ so that
$F$ is an eigenform for $T(p)$ and $\widetilde T_1(p^2)$, and that $m\in\Z_+$ with
$(m,n)=1$.
 If $a(mI)=0$ then
$a(mnI)=0$.
Also, we have
$a(I)a(mnI)=a(mI)a(nI).$

\end{enumerate}

\end{thm}

We also prove the following modest generalization.

\begin{thm}  Suppose that $F$ is a degree 2 Siegel modular form of
weight $k\in\Z_+$, level $\stufe$ and character $\chi$ with Fourier expansion
$$F(\tau)=\sum_T a(T) \exp(2\pi iTr(T\tau)).$$
Suppose that $p$ is an odd prime, and set $D=\begin{pmatrix}1\\&p\end{pmatrix}.$
Let $\calS$ be the set of odd primes so that for $q\in\calS$, 
$F$ is an eigenform for $T(q)$ and $\widetilde T_1(q^2)$, and either $q=p$ or
$\left(\frac{-p}{q}\right)=-1$.
Let $n$ be a product of powers of primes in $\calS$.
Then for any $m\in\Z_+$ so that $(m,n)=1$, we have
$$a(D)a(mnD)=a(mD)a(nD).$$
Also,
$a(D)a(mnD)=0$ if $a(mD)=0$.
\end{thm}

We note that McCarthy applies his results to compute eigenvalues of the level 1 Eisenstein series with regard to the Hecke operators $T(p^r)$ ($p$ prime); as he notes, in \cite{W} we computed the Hecke-eigenvalues of Eisenstein series of square-free levels for all primes $p$, allowing nontrivial character (then generalized in \cite{W2} for arbitrary level $\stufe$ and character $\chi$, but only for primes $p$ so that $p^2\nmid\stufe$).

We further note that it seems that these results cannot be extended to higher degrees, as Lemma 3.1 (which is pivotal for our arguments) does not extend to higher degrees.

\bigskip

\section{Preliminaries}
We will use some language and notation commonly used in quadratic forms  and modular forms theory.  When $\Lambda$ is a lattice whose quadratic form is given by the matrix $T$ (relative to some $\Z$-basis for $\Lambda$), we write $\Lambda\simeq T$.  Now suppose that $\Lambda$ is a lattice with $\Lambda\simeq T$ and that $m\in\Q_+$; we write $\Lambda^m$ to denote the lattice $\Lambda$ ``scaled" by $m$, meaning that $\Lambda^m\simeq mT$.  Also, the discriminant of $\Lambda$ is $\det T$.  With $\Lambda, \Omega$ lattices on the same unlerlying quadratic space over $\Q$, we write $\{\Lambda:\Omega\}$ to denote the invariant factors of $\Omega$ in $\Lambda$.

We set
$$\h_{(2)}=\{X+iY:\ X,Y\in\R^{2,2}_{\sym}:\ Y>0\ \},$$
where $\R^{2,2}_{\sym}$ denotes the set of $2\times 2$ symmetric matrices with real entries, and $Y>0$ means that $Y$ represents a positive definite quadatic form.
For a ring $R$, we write $Sp_2(R)$ for the group of 
$4\times 4$ symplectic matrices with entries in $R$.
Fixing a weight $k\in\Z_+$, for
$\gamma=\begin{pmatrix}A&B\\C&D\end{pmatrix}\in Sp_2(\Q)$, we define
$$F(\tau)|\gamma
=(\det \gamma)^{k/2}\det(C\tau+D)^{-k}F((A\tau+B)(C\tau+D)^{-1}).$$
When  $F$ is a degree 2 Siegel modular form of weight $k$, level $\stufe$ and character $\chi$, this means that for
$\gamma=\begin{pmatrix}A&B\\C&D\end{pmatrix}\in\Gamma_0(\stufe)$, we have
\begin{align*}
F(\tau)|\gamma
=\chi(\det D_\gamma)F(\tau).
\end{align*}
We can write $F$ as a Fourier series:
$$F(\tau)=\sum_{T\ge 0} a(T)\exp(2\pi iTr(T\tau))$$
where the sum is over $2\times 2$ symmetric, positive semi-definite, half-integral matrices $T$ (so the entries in $T$ are half-integers with integers on the diagonal).  
Given $G\in GL_2(\Z)$, we have $\gamma=\begin{pmatrix}G^{-1}\\&^tG\end{pmatrix}\in \Gamma_0(\stufe)$.  Hence
\begin{align*}
\chi(\det G) F(\tau)&=F(\tau)|\gamma\\
&=(\det G)^k F(G^{-1}\tau\,^tG^{-1})\\
&=(\det G)^k \sum_T a(\,^tGTG) \exp(2\pi i Tr(T\tau)).
\end{align*}
Thus $a(^tGTG)=\chi(\det G)(\det G)^k a(T).$ So 
% (as discussed in \cite{HW} when $\chi=1$) 
we can also write $F$ as a ``Fourier series" supported on isometry classes of even integral, positive semi-definite lattices:
For $\Lambda$ an even integral lattice with $\Z$-basis $\{x,y\}$, set $c(\Lambda)=a(T_{\Lambda})$ where, relative to the given basis for $\Lambda$,
we have $\Lambda\simeq2T_{\Lambda}$.   When $\chi(-1)\not=(-1)^k$, we equip $\Lambda$ with an orientation, meaning that with $G\in GL_2(\Z)$, 
$(x\ y)G$ is a basis for the oriented lattice $\Lambda$ if and only if $\det G=1$.  Then
$$F(\tau)=\sum_{\cls\Lambda}c(\Lambda)\e^*\{\Lambda\tau\}$$
where  $\cls \Lambda$ varies over all isometry classes of (oriented) even integral, positive semi-definite lattices, and 
$$\e^*\{\Lambda\tau\}=\sum_G\exp(2\pi iTr(\,^tGT_{\Lambda}G\tau))$$
where 
% $T_{\Lambda}$ an even integral matrix representing the quadratic form on 
% $\Lambda$, 
$G$ varies over
$ O(\Lambda)\backslash GL_2(\Z)$ when $\chi(-1)=(-1)^k$, and   $G$ varies over
$O^+(\Lambda)\backslash SL_2(\Z)$ otherwise.  (Here $O(\Lambda)$ denotes the orthogonal group of $\Lambda$, and $O^+(\Lambda)=O(\Lambda)\cap SL_2(\Z)$.)

Still suppose that  $F$ is a Siegel modular form of degree 2, weight $k$, level $\stufe$ and character $\chi$. 
For $p$ prime, we define $T(p)$, $T_1(p^2)$, and $T_2(p^2)$ as follows.
Take $\delta(p)=\diag(p,p,1,1)$, $\delta_1(p^2)=(p,1,1/p,1)$, and $\delta_2(p^2)=\diag(p,p,1/p,1/p)$.
With $\Gamma=\Gamma_0(\stufe),$ we set
$$F|T(p)=p^{k-3}\sum_{\gamma}\overline\chi(\gamma) F|\delta(p)^{-1}\gamma$$
where $\gamma$ varies over $(\delta(p)\Gamma\delta(p)^{-1}\cap\Gamma)\backslash\Gamma,$
and for $j=1,2$, we set
$$F|T_j(p^2)=p^{j(k-3)}\sum_{\gamma}\overline\chi(\gamma) F|\delta_j(p^2)^{-1}\gamma$$
where $\gamma$ varies over $(\delta_j(p^2)\Gamma\delta_j(p^2)^{-1}\cap\Gamma)\backslash\Gamma.$
Note that replacing $\delta(p)$ or $\delta_j(p^2)$ by a scalar multiple of itself does not change the definition of the associated Hecke operator.
Note also that in \cite{HW}, we did not normalize $T_j(p^2)$ by $p^{j(k-3)}$, as is usually done in other texts, and has been done in the above formula for $T_1(p^2)$.
With $\widetilde T_1(p^2)=T_1(p^2)+\chi(p)p^{k-3}(p+1)$,
Theorem 6.1 of \cite{HW} gives us the following.

\begin{thm}  Let $F$ be a degree 2 Siegel modular form of weight $k$, level $\stufe$, character $\chi$, and lattice coefficients $c(\Lambda)$.  Then for any even integral lattice $\Lambda$, the $\Lambda$th coefficient of $F|T(p)$ is
$$\chi(p^2)p^{2k-3}c(\Lambda^{1/p})
+\chi(p)p^{k-2}\cdot\sum_{\{\Lambda:\Omega\}=(1,p)} c(\Omega^{1/p})
+c(\Lambda^p),$$
and the $\Lambda$th coefficient of $F|\widetilde T_1(p^2)$ is
$$\chi(p^2)p^{2k-3}\cdot\sum_{\{\Lambda:\Omega\}=(1/p,1)} c(\Omega)
+\chi(p)p^{k-2}\alpha(\Lambda;p)c(\Lambda)
+\sum_{\{\Lambda:\Omega\}=(1,p)} c(\Omega).$$
With $Q$ the quadratic form on $\Lambda$, we equip $\Lambda/p\Lambda$ with the quadratic form $\frac{1}{2}Q$, and
 $\alpha(\Lambda;p)$ is the number of isotropic lines in the quadratic space
$\Lambda/p\Lambda$.
There are $p+1$ lines in $\Lambda/p\Lambda$, and each of these
lines is generated either by $y+p\Lambda$ or by  $(x+uy)+p\Lambda$ for some $u$ with $0\le u<p$.
So with $\Lambda\simeq 2I$,
$\alpha(\Lambda;2)=1$, $\alpha(\Lambda;p)=2$ when $p\equiv1\ (4)$, and 
$\alpha(\Lambda;p)=0$ when $p\equiv3\ (4)$.
When $\Lambda\simeq 2T$ with $p|T$, $\alpha(\Lambda;p)=p+1.$

\end{thm}

Note that with $p$ a prime and $m\in\Z_+$ so that $p\nmid m$, for any even integral rank 2 lattice $\Lambda$ we have $\alpha(\Lambda;p)=\alpha(\Lambda^m;p)$ since scaling by $m$ does not change whether a line is  isotropic in $\Lambda/p\Lambda$.

\bigskip

\section{Proof of Theorem 1.1}

The next lemma is pivotal in our proof of Theorem 1.1; when this lemma generalizes, we can generalize this theorem (as seen in Theorem 1.2).

\begin{lem}  Suppose that $F$ is a degree 2 Siegel modular form of weight $k$, level $\stufe$, character $\chi$, and lattice coefficients $c(\Lambda)$.
With $\Delta\simeq 2I$, $p$ prime and $m\in\Z_+$ so that $p\nmid m$, we have
$$\sum_{\{\Delta:\Omega\}=(1/p,1)}c(\Omega^{pm})=\sum_{\{\Delta:\Omega\}=(1,p)}c(\Omega^{m/p})
=\eta(p) c(\Delta^m)$$
where, as in Theorem 1.1,
$$\eta(p)=\begin{cases} 1+\chi(-1)(-1)^k&\text{if $p\equiv1\ (4)$,}\\
0&\text{if $p\equiv 3\ (4)$,}\\
1&\text{if $p=2.$}
\end{cases}$$
% In particular, $\eta(p)$ is only dependent on $p$.
\end{lem}

\begin{proof}  
Suppose that $\{\Delta:\Omega\}=(1/p,1)$.  Then $\{\Delta:p\Omega\}=(1,p)$; also, with $T$ a matrix so that $\Omega^{m/p}\simeq \frac{m}{p}T$, we have
$p\Omega^{m/p}\simeq pm T$.  This proves that 
$$\sum_{\{\Delta:\Omega\}=(1/p,1)}c(\Omega^{pm})=\sum_{\{\Delta:\Omega\}=(1,p)}c(\Omega^{m/p}).$$

Let $\{x,y\}$ be a basis for $\Delta$ relative to which $\Delta\simeq 2I$, 
and suppose that $\{\Delta:\Omega\}=(1,p)$.
Thus $\Omega=\Z(x+uy)\oplus\Z py$ for $0\le u<p$ or $\Omega=\Z px\oplus\Z y$.
Hence $\Omega^{m/p}$ is even integral if and only if $\Omega=\Z(x+u y)\oplus\Z py$
with $u^2\equiv -1\ (p).$
If $p\equiv 3\ (4)$, there are no such $u$.
Suppose that $p\equiv 1\ (4)$, and fix $u$ so that $u^2\equiv-1\ (p)$.
Set $\Omega_u=\Z(x+u y)\oplus\Z py$ and $\Omega_{-u}=\Z(x-u y)\oplus\Z py$.  Then $\Omega_u^{1/p}$ and $\Omega_{-u}^{1/p}$ are integral with determinant 1.
Thus by Exercise 5 p. 77 of \cite{N}, there is some $G\in GL_2(\Z)$ so that $^tGTG=I$.  Therefore $c(\Omega_u^{m/p})=\chi(\det G)(\det G)^kc(\Delta^m).$
When $p=2$, $\Omega^{m/2}$ is even integral only for $\Omega_1=\Z(x+y)\oplus \Z 2y\simeq2\begin{pmatrix}1&1\\1&2\end{pmatrix}$.  Since $\,^tG\begin{pmatrix}1&1\\1&2\end{pmatrix}G=I$
for $G=\begin{pmatrix}1&-1\\0&1\end{pmatrix}$, we have
$c(\Omega_1^{1/2})=c(\Delta).$
Thus when $p=2$, the sum on $\Omega$ is
$c(\Omega^{m/2})=c(\Delta^m).$
\end{proof}

In the next proposition we use Lemma 3.1 to establish some very useful identities.

\begin{prop}  Suppose that $F$ is a degree 2 Siegel modular form of weight $k$, level $\stufe$, character $\chi$, and
and lattice coefficients $c(\Lambda)$.
Also suppose that 
$F|T(p)=\lambda(p)F$ and $F|\widetilde T_1(p^2)=\widetilde\lambda_1(p^2)F$.
Set $\eta(1)=0$, $\kappa(1)=1$.
With $\Delta\simeq 2I$ and $m\in\Z_+$ so that $p\nmid m$, for  $r\ge 1$ 
we inductively define $\eta(p^r)$ and $\kappa(p^r)$ as follows:
$\eta(p)$ is as in Proposition 3.1, $\kappa(p)=\lambda(p)-\chi(p)p^{k-2}\eta(p)$,
and for $r\ge 2$, 
$$\eta(p^r)=\widetilde\lambda_1(p^2)\kappa(p^{r-2})
-\chi(p^2)p^{2k-3}\eta(p^{r-2})-\chi(p)p^{k-2}\alpha(\Delta^{p^{r-2}};p)\kappa(p^{r-2})$$
and
$$\kappa(p^r)=\lambda(p)\kappa(p^{r-1})
-\chi(p^2)p^{2k-3}\kappa(p^{r-2})-\chi(p)p^{k-2}\eta(p^r).$$
Then we have
\begin{align}
\sum_{\{\Delta:\Omega\}=(1/p,1)} c(\Omega^{p^r m})
=\sum_{\{\Delta:\Omega\}=(1,p)} c(\Omega^{p^{r-2}m})
=\eta(p^{r}) c(\Delta^m)
\end{align}
and
\begin{align}
c(\Delta^{p^{r}m})=\kappa(p^{r}) c(\Delta^m).
\end{align}
\end{prop}

\begin{proof} Recall that the value of $\alpha(\Delta;p)$ is computed after Theorem 2.1; note that for $r\ge1$, $\alpha(\Delta^{p^r};p)=p+1$ as then $\Delta^{p^r}/p\Delta^{p^r}$ is totally isotropic and contains $p+1$ lines.
Also, note that the first equality in Equation (1) is easily verified by replacing $\Omega$ by $p\Omega$.  We now compute $\eta(p^r)$ and $\kappa(p^r)$.

(Case $r=0$:)  With $\kappa(1)=1$, it is clear that $c(\Delta)=\kappa(1)c(\Delta)$.
 So suppose that we have $\{\Delta:\Omega\}=(1,p)$.  Then $\disc\Omega^{m/p^2}=4m^2/p^2$.  Hence when $p\not=2$, $\Omega^{m/p}$ cannot be integral, so $c(\Omega^{m/p})=0$.  
When $p=2$, we see from the discussion at the end of the proof of Lemma 3.1 that $\Omega^{m/4}$ is not even integral for any $\Omega$ with 
$\{\Delta:\Omega\}=(1,2).$
Thus Equation (1) holds with $\eta(1)=0$.

(Case $r=1$:)
In Lemma 3.1 we showed that 
Equation (1) holds with $\eta(p)$ as defined therein.
We know that $c(\Delta^{m/p})=0$ since $\Delta^{m/p}$ is not even integral, and so
by Theorem 2.1 and the above conclusion we have
$$\kappa(p)c(\Delta^m)\lambda(p)c(\Delta^m)-\chi(p)p^{k-2}\eta(p)c(\Delta^m).$$

(Induction step:)  Suppose that $r\ge 2$ and that the proposition holds for all $\ell$ with $0\le \ell<r$.  First, from Theorem 2.1 and the induction hypothesis
we have
\begin{align*}
\sum_{\{\Delta:\Omega\}=(1,p)}c(\Omega^{p^{r-2}m})
&=(\widetilde\lambda_1(p^2)\kappa(p^{r-2})-
\chi(p^2)p^{2k-3}\eta(p^{r-2}))c(\Delta^m)\\
&\quad-\chi(p)p^{k-2}\alpha(\Delta^{p^{r-2}};p)\kappa(p^{r-2})c(\Delta^m)\\
&=\eta(p^r)c(\Delta^m).
\end{align*}
Hence we also have
\begin{align*}
c(\Delta^{p^rm})
&=(\lambda(p)\kappa(p^{r-1})-\chi(p^2)p^{2k-3}\kappa(p^{r-2})
-\chi(p)p^{k-2}\eta(p^r))c(\Delta^m)\\
&=\kappa(p^r)c(\Delta^m).
\end{align*}
Thus induction on $r$ proves the proposition.
\end{proof}

We also have the following helpful result.

\begin{prop}  Suppose that $F$ is a degree 2 Siegel modular form of weight $k$, level $\stufe$, character $\chi$, and lattice coefficients $c(\Lambda)$; recall that $c(\Lambda)=a(T_{\Lambda})$ where $\Lambda\simeq2T_{\Lambda}$.
Fix a prime $p$ and $r\ge 1$; take
 $\Delta\simeq 2I$ relative to a $\Z$-basis $\{x,y\}$.  
Set $\epsilon=1+\chi(-1)(-1)^k$.
Then with $\eta(p)$ as defined in Lemma 3.1 and $\eta(p^{r+1})$ as defined in Proposition 3.2, we have
\begin{align*}
&\eta(p)a(p^rI)-\eta(p^{r+1})a(I)\\
&\quad =
-\epsilon  a\begin{pmatrix}p^{r-1}m\\&p^{r+1}m\end{pmatrix}
-\epsilon \sum_{\substack{1\le u<p/2\\ u^2\not\equiv-1\,(p)}}
a\left(p^rm\begin{pmatrix}(1+u^2)/p&u\\u&p\end{pmatrix}\right).
\end{align*}
\end{prop}

\begin{proof}  
By Proposition 3.2,
$\eta(p^{r+1})c(\Delta)=\sum_{\{\Delta:\Omega\}=(1,p)}c(\Omega^{p^{r-1}}).$
With $\Omega$ so that $\{\Delta:\Omega\}=(1,p)$, we either have $\Omega=\Z(x+uy)\oplus\Z py$ for $0\le u< p$, or $\Omega=\Z px\oplus\Z y$.  
Then for $u\not=0$, we have
$\Omega_u=\Z(x+uy)\oplus\Z py\simeq 2p^{r+1}\begin{pmatrix}((1+u^2)/p&u\\u&p\end{pmatrix};$ from our above discussion on Fourier coefficients of a Siegel modular form $F$, we have 
$c(\Omega_u^{1/p})=\chi(-1)(-1)^kc(\Omega_{-u}^{1/p}).$
Similarly, 
$$c((\Z px\oplus\Z y)^{1/p})=\chi(-1)(-1)^kc((\Z x\oplus\Z py)^{1/p}).$$
Further, if $p$ is odd and $u^2\equiv -1\ (p)$, then by Exercise 5 p. 77 of \cite{N},
there is some $G\in GL_2(\Z)$ so that
$$^tG\begin{pmatrix}(1+u^2)/p&u\\u&p\end{pmatrix}G=I;$$
hence with $G'=\diag(-1,1)G$, we get
$$^tG'\begin{pmatrix}(1+u^2)/p&-u\\-u&p\end{pmatrix}G'=I,$$
and thus
$c(\Omega_u^{1/p})+c(\Omega_{-u}^{1/p})
=(1+\chi(-1)(-1)^k)c(\Delta^{p^r}).$
Similarly, when $p=2$, $\Omega_1\simeq 2^{r+2}m\begin{pmatrix}1&1\\1&2\end{pmatrix},$
which can be diagonalized using the matrix 
$G=\begin{pmatrix}1&-1\\0&1\end{pmatrix},$  
and so $c(\Omega_1^{1/2})=c(\Delta^{p^r}).$
Using the definition of $\eta(p)$, the proposition now follows.
\end{proof}

Theorem 1.1 is now easy to prove.  Take $\Delta\simeq 2I$; recall that
$c(\Delta^{p^rm})=a(p^rmI).$
The first claim of (a) follows immediately from Theorem 2.1 and Lemma 3.1.
To prove the second claim in (a), we first use Theorem 2.1 to get
\begin{align}
\widetilde\lambda_1(p^2)c(\Delta^m)
=\chi(p)p^{k-2}\alpha(\Delta;p)c(\Delta^m)
+\sum_{\{\Delta:\Omega\}=(1,p)}c(\Omega)
\end{align}
and
\begin{align}
\lambda(p)c(\Delta^{pm})
&= \chi(p^2)p^{2k-3}c(\Delta)
%\\&\quad 
+\chi(p)p^{k-2}\sum_{\{\Delta:\Omega\}=(1,p)}c(\Omega)
+c(\Delta^{p^2}).
\end{align}
Solving Equation (4) for the sum on $\Omega$ and substituting into 
$\chi(p)p^{k-2}\cdot$Equation (3) yields the second claim in (a).

To prove (b), we first use Theorem 2.1 and Proposition 3.2 to obtain
\begin{align*}
a({p^{r+1}}I)
&=\lambda(p)a({p^r}I)-\chi(p^2)p^{2k-3}a({p^{r-1}}I)\\
&\quad -\chi(p)p^{k-2}\eta(p^{r+1})a(I).
\end{align*}
Next we multiply this equation by $a(mI)$, use Theorem 1.1(a) to substitute for $\lambda(p)a(mI)$, and use  Proposition 3.3 to substitute for 
$\eta(p)a(p^rI)-\eta(p^{r-1}I)a(I)$; (b) now immediately follows.

For (c), suppose that $n=p_1^{e_1}\cdots p_t^{e_t}$ where $p_1,\ldots,p_t$ are distinct primes so that $F$ is an eigenform for $T(p_i)$ and $\widetilde T_1(p_i^2)$
($1\le i\le t$).
For any $m'\in\Z_+$ with $(n,m')=1$,
repeated applications of Proposition 3.2 gives us
$$a(m'nI)=\kappa(p_1^{e_1})\cdots \kappa(p_t^{e_t})a(m'I).$$
Thus (taking $m'=m$) we have $a(mnI)=0$ if $a(mI)=0$.
Further (taking $m'=1$), we have
$$a(nI)=\kappa(p_1^{e_1})\cdots \kappa(p_t^{e_t})a(I)$$
and hence
$a(I)a(mnI)=a(mI)\kappa(p_1^{e_1})\cdots\kappa(p_t^{e_t})a(I)=a(mI)a(nI).$

\bigskip

\section{Proof of Theorem 1.2}

As previously noted, the key to proving Theorem 1.1 is Lemma 3.1.  We can extend this lemma to some extent, as follows.

\begin{lem}  
 Suppose that $F$ is a degree 2 Siegel modular form of weight $k$, level $\stufe$, and character $\chi$, and let $c(\Lambda)$ denote the $\Lambda$th coefficient of $F$.  
Suppose that $p$ is an odd prime and $\Delta\simeq 2\begin{pmatrix}1\\&p\end{pmatrix}.$ 
For $m\in\Z_+$ with $p\nmid m$, we have
$$\sum_{\{\Delta:\Omega\}=(1/p,1)}c(\Omega^{pm})=\chi(-1)(-1)^k c(\Delta^m).$$
For $q$ an odd prime with 
$\left(\frac{-p}{q}\right)=-1$ and $q\nmid m$, we have
$$\sum_{\{\Delta:\Omega\}=(1/q,1)}c(\Omega^{qm})=0.$$
\end{lem}

\begin{proof}  Let $\{x,y\}$ be a $\Z$-basis for $\Delta$ relative to which
$\Delta\simeq \begin{pmatrix}2\\&2p\end{pmatrix}.$ 
Then the only lattice $\Omega$ so that $\{\Delta:\Omega\}=(1,p)$ and $\Omega^{m/p}$ is even integral if
$$\Omega=\Z px\oplus\Z y\simeq 2p\begin{pmatrix}p\\&1\end{pmatrix}
=2p\begin{pmatrix}0&1\\1&0\end{pmatrix}\begin{pmatrix}1\\&p\end{pmatrix} 
\begin{pmatrix}0&1\\1&0\end{pmatrix}.$$
With $\gamma=\diag\left(\begin{pmatrix}0&1\\1&0\end{pmatrix},\begin{pmatrix}0&1\\1&0\end{pmatrix}\right)$, we have $F|\gamma=\chi(-1)F$ and consequently $c\left(2m\begin{pmatrix}p\\&1\end{pmatrix} \right)
=\chi(-1)(-1)^k c\left(2m\begin{pmatrix}1\\&p\end{pmatrix} \right).$
Hence 
$$\sum_{\{\Delta:\Omega\}=(1/p,1)}c(\Omega^{pm})=
\sum_{\{\Delta:\Omega\}=(1,p)}c(\Omega^{m/p})=\chi(-1)(-1)^k c(\Delta^m).$$

With $q$ an odd prime with $\left(\frac{-p}{q}\right)=-1$ and $q\nmid m$, there is no lattice $\Omega$ so that 
$\{\Delta:\Omega\}=(1,q)$ and $\Omega^{m/q}$ is even integral, and hence
$$\sum_{\{\Delta:\Omega\}=(1/q,1)}c(\Omega^{qm})=0.$$
\end{proof}

To prove Theorem 1.2, we begin by making the following definitions.
Set $\eta(1)=0$, $\kappa(1)=1$.
For $q\in \calS$ (as defined in the statement of Theorem 1.2), define $\eta(q)$ as in Lemma 4.1, and set
$\kappa(q)=\lambda(q)-\chi(q)q^{k-2}\eta(q).$  For $r\ge 2$, we define $\eta(q^r)$ and $\kappa(q^r)$ using the inductive formulas from Proposition 3.2 (so $\eta(q^r), \kappa(q^r)$ are determined by $\eta(q)$, $\lambda(q)$ and $\widetilde\lambda_1(q^2)$).
Then mimicking the proofs of Proposition 3.2 and Theorem 1.1(c) easily yields Theorem 1.2.

\end{document}